\documentclass[12pt]{amsart}
\usepackage[all]{xy}

\usepackage{graphicx}  

\usepackage{amsmath}
\usepackage{amssymb}
\usepackage{amsfonts}
\usepackage{mathrsfs}
\usepackage{mathtools}

\newcommand{\Cdb}{\ensuremath{\mathbb{C}}}

\newcommand{\A}{\mbox{${\mathcal A}$}}

\newcommand{\norm}[1]{\Vert#1\Vert}
\newcommand{\bignorm}[1]{\bigl\Vert#1\bigr\Vert}
\newcommand{\Bignorm}[1]{\Bigl\Vert#1\Bigr\Vert}

\newcommand{\hten}{\mathop{\otimes}\limits^{2}}

\newtheorem{theorem}{Theorem}[section]
\newtheorem{lemma}[theorem]{Lemma}

\newtheorem{proposition}[theorem]{Proposition}

\theoremstyle{remark}

\theoremstyle{definition}

\numberwithin{equation}{section}

\begin{document}

\title[]{A sharp equivalence between $H^\infty$ functional calculus and square 
function estimates}

\author{Christian Le Merdy}
\address{Laboratoire de Math\'ematiques\\ Universit\'e de  Franche-Comt\'e
\\ 25030 Besan\c con Cedex\\ France}
\email{clemerdy@univ-fcomte.fr}

\date{\today}

\begin{abstract} Let $(T_t)_{t\geq 0}$
be a bounded analytic semigroup on $L^p(\Omega)$, with $1<p<\infty$. Let $-A$ 
denote its infinitesimal generator.
It is known that if $A$ and $A^*$ both satisfy square function estimates 
$\bignorm{\bigl(\int_{0}^{\infty}\vert A^{\frac{1}{2}}
T_t(x)\vert^2\, dt\,\bigr)^{\frac{1}{2}}}_{L^p}
\,\lesssim\,\norm{x}_{L^p}$ and 
$\bignorm{\bigl(\int_{0}^{\infty}\vert A^{*\frac{1}{2}}
T_t^*(y)\vert^2\, dt\,\bigr)^{\frac{1}{2}}}_{L^{p'}}
\,\lesssim\,\norm{y}_{L^{p'}}$ for $x\in L^p(\Omega)$ and $y\in L^{p'}(\Omega)$,
then $A$ admits a bounded $H^{\infty}(\Sigma_\theta)$ functional 
calculus for any $\theta>\frac{\pi}{2}$. We show that this actually holds true for
some $\theta<\frac{\pi}{2}$.
\end{abstract}

\maketitle


\bigskip

\section{Introduction}
Let $(\Omega,\mu)$ be a measure space, let $1<p<\infty$ and let $(T_t)_{t\geq 0}$
be a bounded analytic semigroup on $L^p(\Omega)$. Let $-A$ denote its infinitesimal
generator. Various square functions can be associated to $(T_t)_{t\geq 0}$ and $A$.
In this paper we will focus on the following ones. For any positive real number $\alpha>0$
and any $x\in L^p(\Omega)$, we set
\begin{equation}\label{Alpha}
\norm{x}_{A,\alpha}\,=\,\Bignorm{\Bigl(\int_{0}^{\infty} t^{2\alpha-1}
\bigl\vert A^{\alpha}T_t(x)\bigr\vert^2
\, dt\,\Bigr)^{\frac{1}{2}}}_{L^p(\Omega)}.
\end{equation}
Such square functions have played a key role in the analysis of analytic semigroups
and in various estimates involving them for at least 40 years. 
In \cite{MI} and \cite{CDMY}, McIntosh and his co-authors introduced $H^\infty$ functional 
calculus and gave remarkable connections
between boundedness of that functional 
calculus and estimates of square functions on $L^p$-spaces. The aim
of this note is to give an improvement of their main result regarding the angle 
condition in the
$H^\infty$ functional calculus.

\bigskip
We start with a little background on sectoriality and $H^\infty$ functional calculus.
We refer to \cite{CDMY, H, KW1, LM1, LM3, MI} for details and complements. For any angle
$\omega\in(0,\pi)$, we consider the sector $\Sigma_\omega=\{z\in\Cdb^*\, :\, 
\vert{\rm Arg}(z)\vert<\omega\}$. 
Let $X$ be a Banach space and let $B(X)$ denote the algebra of all bounded operators
on $X$. A closed, densely defined linear operator $A\colon D(A)\subset X\to X$
is called sectorial of type $\omega$ if its spectrum $\sigma(A)$ is included in
the closed sector $\overline{\Sigma_\omega}$, and for any angle
$\theta\in(\omega,\pi)$, there is a constant $K_\theta>0$ such that 
\begin{equation}\label{Sector}
\vert  z \vert\norm{(z -A)^{-1}}\leq K_\theta,\qquad z
\in\Cdb\setminus \overline{\Sigma_\theta}.
\end{equation} 
We recall that sectorial operators of type $<\frac{\pi}{2}$ coincide with 
negative generators of bounded analytic semigroups.

For any $\theta\in(0,\pi)$, let $H^{\infty}(\Sigma_\theta)$ be the algebra of all bounded 
analytic functions $f\colon \Sigma_\theta\to \Cdb$, equipped with the supremum norm
$\norm{f}_{\infty,\theta}=\,\sup\bigl\{\vert f(z)\vert \, :\, z\in \Sigma_\theta\bigr\}.$
Let $H^{\infty}_{0}(\Sigma_\theta)\subset H^{\infty}(\Sigma_\theta)$ be the subalgebra
of bounded 
analytic functions $f\colon \Sigma_\theta\to \Cdb$ for which there exist $s,c>0$
such that $\vert f(z)\vert\leq c \vert z \vert^s\bigl(1+\vert z \vert)^{-2s}\,$ 
for any $z\in \Sigma_\theta$. Given a sectorial operator $A$ of type $\omega\in(0,\pi)$, 
a bigger angle 
$\theta\in(\omega,\pi)$, and a function $f\in H^{\infty}_{0}(\Sigma_\theta)$,
one may define a bounded operator $f(A)$ by means of a Cauchy integral (see e.g. \cite[Section 2.3]{H}
or \cite[Section 9]{KuW});
the resulting
mapping $H^{\infty}_{0}(\Sigma_\theta)\to B(X)$ taking $f$ to $f(A)$ is
an algebra homomorphism. By definition, $A$ has a bounded $H^{\infty}(\Sigma_\theta)$
functional calculus provided that this homomorphism is bounded, that is, there exists
a constant $C>0$ such that $\norm{f(A)}_{B(X)}\leq C\norm{f}_{\infty,\theta}$ for any 
$f\in H^{\infty}_{0}(\Sigma_\theta)$. In the case when $A$ has a dense range, 
the latter boundedness condition allows a natural extention of
$f\mapsto f(A)$ to the full algebra $H^{\infty}(\Sigma_\theta)$.

It is clear that if $\omega<\theta_1<\theta_2<\pi$ and the operator $A$ admits a
bounded $H^{\infty}(\Sigma_{\theta_1})$ functional calculus, then it 
admits a bounded $H^{\infty}(\Sigma_{\theta_2})$ functional calculus as well.
Indeed we have $H^{\infty}(\Sigma_{\theta_2})\subset H^{\infty}(\Sigma_{\theta_1})$,
with $\norm{\cdotp}_{\infty,\theta_1}\leq\norm{\cdotp}_{\infty,\theta_2}$.
However the converse is wrong, as shown by Kalton \cite{K}. The resulting issue of reducing the 
angle of a bounded  $H^{\infty}$ functional calculus is at the heart of the
present work.

\bigskip
Let us now focus on the case when $X=L^p(\Omega)$, with $1<p<\infty$.
For simplicity we let $\norm{\ }_p$ denote the norm on this space.
We let $p'=p/(p-1)$ denote the conjugate number of $p$.
Assume that $A\colon D(A)\subset L^p(\Omega)
\to L^p(\Omega)$ is a sectorial operator of
type $<\frac{\pi}{2}$, and let $T_t=e^{-tA}$ for $t\geq 0$. 
It follows from \cite[Cor. 6.7]{CDMY} that
if $A$ has a bounded $H^{\infty}(\Sigma_\theta)$ 
functional calculus for some $\theta<\frac{\pi}{2}$, then for any $\alpha>0$, it satisfies a
square function estimate 
$$
\norm{x}_{A,\alpha}\lesssim\norm{x}_p,\qquad x\in L^p(\Omega).
$$ 
Further, the adjoint $A^*$ is a sectorial operator 
of type $\omega$ on $X^*=L^{p'}(\Omega)$ and $f(A)^*=f(A^*)$ for any
$f\in H^{\infty}_{0}(\Sigma_\theta)$. Thus $A^*$ admits a bounded $H^{\infty}(\Sigma_\theta)$
functional calculus as well hence  
satisfies estimates $\norm{y}_{A^*,\alpha}\lesssim\norm{y}_{p'}$
for $y\in L^{p'}(\Omega)$. 

Conversely, it follows from \cite[Cor. 6.2]{CDMY}
and its proof that if $A$ and $A^*$ satisfy square function 
estimates $\norm{x}_{A,\frac{1}{2}}\lesssim\norm{x}_p$ and 
$\norm{y}_{A^*,\frac{1}{2}}\lesssim\norm{y}_{p'}$,
then $A$ admits a bounded $H^{\infty}(\Sigma_\theta)$ functional calculus
for any $\theta>\frac{\pi}{2}$. 
Our main result is that the latter property
can be achieved for some $\theta<\frac{\pi}{2}$. Altogether, 
this yields the following equivalence statement.

\begin{theorem}\label{Main}
Let $1<p<\infty$ and let $(T_t)_{t\geq 0}$ be a 
bounded analytic semigroup on $L^p(\Omega)$, with generator $-A$.
The following assertions are equivalent.
\begin{itemize}
\item [(i)] $A$ and $A^*$ satisfy square function estimates
\begin{equation}\label{SFE}
\norm{x}_{A,\frac{1}{2}}\lesssim\norm{x}_p
\qquad\hbox{and}\qquad
\norm{y}_{A^*,\frac{1}{2}}\lesssim\norm{y}_{p'}
\end{equation}
for $x\in L^p(\Omega)$ and $y\in L^{p'}(\Omega)$.
\item [(ii)] There exists $\theta\in \bigl(0,\frac{\pi}{2}\bigr)$ such that
$A$ admits a bounded $H^{\infty}(\Sigma_\theta)$ functional calculus.
\end{itemize}
\end{theorem}

In the above discussion and later on in the paper, we write 
`$N_1(x)\lesssim N_2(x)$ for $x\in X$' to indicate an inequality
$N_1(x)\leq C N_2(x)$ which holds for a constant $C>0$ not depending
on $x\in X$.

It follows from the pionneering work of Dore and Venni \cite{DV} 
that if $A$ has a bounded $H^{\infty}(\Sigma_\theta)$ functional calculus
for some $\theta<\frac{\pi}{2}$, then it has the so-called maximal
$L^q$-regularity for $1<q<\infty$ (see e.g. \cite{KuW} for details). As a consequence 
of Theorem \ref{Main}, we see that $A$ has  maximal
$L^q$-regularity provided that $A$ and $A^*$ satisfy the square
function estimates (\ref{SFE}).

More observations will be given in Section 4. The implication `(i)$\,\Rightarrow\,$(ii)'
of Theorem \ref{Main} is proved in Section 3 (see the two lines following Theorem
\ref{Main2}). It uses preliminary
results and some background discussed in Section 2 below.

\section{Preparatory results}

Throughout we let $(\Omega,\mu)$ be a measure space and we fix a number $1<p<\infty$.

We start with a few observations on tensor products. 
Given any  Banach space $Z$, we regard the algebraic tensor
product $L^p(\Omega) \otimes Z$ as a (dense) subspace of the Bochner space
$L^p(\Omega; Z)$ in the usual way. It is plain that for any $S\in B(Z)$,
the tensor extension $I_{L^p}\otimes S$ of $S$ defined on $L^p(\Omega) \otimes Z$ extends to a
bounded operator $I_{L^p}\overline{\otimes} S\colon L^p(\Omega;Z) \to L^p(\Omega;Z)$,
whose norm is equal to $\norm{S}$. 

Let $K$ be a Hilbert space. It is well-known that similarly for any
$T\in B(L^p(\Omega))$, the tensor extension $T\otimes I_K$ extends to a
bounded operator $T\overline{\otimes} I_K\colon L^p(\Omega;K) \to L^p(\Omega;K)$,
whose norm is equal to $\norm{T}$. 

The following is an extension of the latter result. 

\begin{lemma}\label{Tensor}
Let $H,K$ be Hilbert spaces, and let $H\hten K$ denote their Hilbertian 
tensor product. For any bounded operator $T\colon L^p(\Omega)\to L^p(\Omega;H)$,
the tensor extension $T\otimes I_K$ extends to a bounded operator
$$
T\overline{\otimes}I_K\colon L^p(\Omega;K)\longrightarrow L^p(\Omega;H\hten K),
$$
whose norm is equal to $\norm{T}$.
\end{lemma}

This can be easily shown by adapting the proof 
of the scalar case (i.e. $H=\Cdb$) given in \cite[Chap. V; Thm. 2.7]{GCRF}. Details are left to the reader.

Let $(\Lambda,\nu)$ be an auxiliary measure space and recall the
isometric duality isomorphism
$$
L^p(\Omega; L^2(\Lambda))^*\,=\, L^{p'}(\Omega; L^2(\Lambda)),
$$
that we will use throughout without further reference.

Following \cite[Def. 2.7]{JLX},
we say that an element $u$ of $L^p(\Omega; L^2(\Lambda))$ is represented
by a measurable function $\varphi\colon \Lambda\to L^p(\Omega)$ provided that 
$\langle \varphi(\cdotp),y\rangle$ belongs to $L^2(\Lambda)$ for any
$y\in L^{p'}(\Omega)$ and 
$$
\langle u, y\otimes b\rangle=\,\int_{\Lambda} \langle\varphi(t),y\rangle\, b(t)\, d\nu(t),
\qquad y\in L^{p'}(\Omega),\, b\in L^{2}(\Lambda).
$$
Such a reprentation is necessarily unique and
$$
\Bignorm{\Bigl(\int_{\Lambda}\bigl\vert\varphi(t)\bigr\vert^2\, 
d\nu(t)\,\Bigr)^{\frac{1}{2}}}_{L^p(\Omega)}
$$
is the norm of $u$ in $L^p(\Omega; L^2(\Lambda))$. In this case we simply say that 
$\varphi$ belongs to $L^p(\Omega; L^2(\Lambda))$, and make no difference between $u$ and $\varphi$.
If $1<p<2$, $L^p(\Omega; L^2(\Lambda))\subset L^2(\Lambda; L^p(\Omega))$ contractively,
hence every element of $L^p(\Omega; L^2(\Lambda))$ can be represented
by a measurable function $\Lambda\to L^p(\Omega)$. However this is no longer
true if $p>2$, as shown in \cite[App. B]{JLX}.

\bigskip

We let $J=(0,\infty)$ equipped with Lebesgue measure $dt$. The above discussion applies
to the definition of square functions. Namely consider $(T_t)_{t\geq 0}$ and $A$ as in (\ref{Alpha}),
and let $\alpha>0$ and $x\in L^p(\Omega)$. The function $\varphi\colon J\to L^p(\Omega)$
defined by $\varphi(t)=t^{\alpha -\frac{1}{2}} A^\alpha T_t(x)$ is continuous.
When it belongs to $L^p(\Omega;L^2(J))$, then $\norm{x}_{A,\alpha}$ is equal to its norm in 
that space. Otherwise we have $\norm{x}_{A,\alpha}=\infty$.

The following lemma will be used in the analysis of square functions.

\begin{lemma}\label{FVarphi}
Let $\Gamma\colon J\to B(L^p(\Omega))$ be a continuous function such that 
$\Gamma(\cdotp)x$ represents an element of  $L^p(\Omega;L^2(J))$ for any $x\in L^p(\Omega)$,
and there exists a constant $C\geq 0$ such that
$$
\Bignorm{\Bigl(\int_{0}^{\infty}\bigl\vert \Gamma(t)x\bigr\vert^2\, dt\,\Bigr)^{\frac{1}{2}}}_{L^p(\Omega)}
\,\leq\,C\norm{x}_p,\qquad x\in L^p(\Omega).
$$
Let $\varphi\colon J\to L^p(\Omega)$ be a continuons function representing an element of 
$L^p(\Omega;L^2(J))$. Then the function $(s,t)\mapsto \Gamma(t)\varphi(s)$ represents an element
of $L^p(\Omega;L^2(J^2))$ and 
\begin{equation}\label{FVarphi1}
\Bignorm{\Bigl(\int_{0}^{\infty}\int_{0}^{\infty}\bigl\vert \Gamma(t)\varphi(s)\bigr\vert^2\, dsdt \,\Bigr)^{\frac{1}{2}}}_{L^p(\Omega)}\,
\leq C\,\Bignorm{\Bigl(\int_{0}^{\infty}\bigl\vert
\varphi(s)\bigr\vert^2\, ds\,\Bigr)^{\frac{1}{2}}}_{L^p(\Omega)}.
\end{equation}
\end{lemma}

\begin{proof}
According to the assumption on $\Gamma$, we may define a bounded operator
$T\colon L^p(\Omega)\to L^p(\Omega;L^2(J))$ by letting $T(x)= \Gamma(\cdotp)x$ for any 
$x\in L^p(\Omega)$. Recall that we have a natural identification $L^2(J)\hten L^2(J)\simeq
L^2(J^2)$. Hence applying Lemma \ref{Tensor}, we obtain an extension
\begin{equation}\label{Tensor-T}
T\overline{\otimes} I_{L^2(J)}\colon L^p(\Omega;L^2(J))\longrightarrow L^p(\Omega; L^2(J^2)),
\end{equation}
whose norm is equal to $\norm{T}$.

Let $\A\subset L^2(J)$ be the (dense) subspace of all square summable functions
with compact support. For any $y\in L^{p'}(\Omega)$, the function
$(s,t)\mapsto \langle \Gamma(t)\varphi(s),y\rangle$ is continuous hence its product
with any element of $\A\otimes\A$ is integrable. We claim that for any
$\phi\in \A\otimes\A$,
\begin{equation}\label{FVarphi2}
\int_{J^2} \langle \Gamma(t)\varphi(s),y\rangle\, \phi(s,t)\, dsdt\, = 
\bigl\langle (T\overline{\otimes} I_{L^2(J)})(\varphi), 
y\otimes\phi\bigr\rangle.
\end{equation}
To prove this, consider $\phi\in \A\otimes\A$; one
can find finite families
$(a_i)_i$ and $(b_i)_i$ in $\A$ such that $(a_i)_i$ is an orthonormal family, and
$$
\phi=\sum_i a_i\otimes b_i.
$$
Since $\varphi$ is continuous, each $a_i\varphi\colon J\to L^p(\Omega)$ is integrable
and we may define
$$
z_i=\int_{J} a_i(s) \varphi(s)\, ds
$$
in $L^p(\Omega)$. Then we have
\begin{align*}
\int_{J^2}  \langle \Gamma(t)\varphi(s),y\rangle\, \phi(s,t)\, dsdt\,
& =\, \int_{J^2} \sum_i \langle \Gamma(t)\varphi(s),y\rangle\,a_i(s)b_i(t)\, dsdt\\
& = \, \int_{J}\sum_i \langle \Gamma(t)z_i,y\rangle\, b_i(t)\, dt\\
& =\, \sum_i \langle T(z_i), y\otimes b_i\rangle\\
& =\,  \Bigl\langle \sum_i T(z_i)\otimes \overline{a_i}, \sum_i 
y\otimes a_i\otimes b_i\Bigr\rangle.
\end{align*}
Let $Q\colon L^2(J)\to L^2(J)$ denote the orthogonal projection onto the linear span
of the $\overline{a_i}$'s. Then we have 
$$
\sum_i z_i\otimes 
\overline{a_i}=\bigl(I_{L^p}\overline{\otimes} Q\bigr)(\varphi).
$$
Hence the above calculation shows that
\begin{align*}
\int_{J^2}  \langle \Gamma(t)\varphi(s),y\rangle\, \phi(s,t)\, dsdt\, 
& =
\bigl \langle (T\otimes I_{L^2(J)})\bigl(I_{L^p}\overline{\otimes} Q\bigr)(\varphi), 
y\otimes\phi\bigr\rangle\\
& = 
\bigl\langle\bigl(I_{L^p}\overline{\otimes}Q \overline{\otimes} I_{L^2(J)} \bigr)
(T\overline{\otimes} I_{L^2(J)})(\varphi), y\otimes\phi\bigr\rangle\\
& = 
\bigl\langle (T\overline{\otimes} I_{L^2(J)})(\varphi), 
y\otimes (Q^*\otimes I_{L^2(J)})(\phi)\bigr\rangle.
\end{align*}
Moreover in the duality considered here, $(Q^*\otimes I_{L^2(J)})(\phi)=\phi$, 
hence we obtain (\ref{FVarphi2}).

The latter identity shows that 
$$
\Bigl\vert \int_{J^2}  \langle \Gamma(t)\varphi(s),y\rangle\, 
\phi(s,t)\, dsdt\,\Bigr\vert\leq\norm{T}\norm{y}_{p'}
\norm{\phi}_{L^2(J^2)}\,\norm{\varphi}_{L^p(\Omega;L^2(J))}.
$$
By the density of $\A\otimes \A$ in $L^2(J^2)$ and duality, this shows that
$(s,t)\mapsto \langle \Gamma(t)\varphi(s),y\rangle\,$ belongs to $L^2(J^2)$. 
By densiy again, (\ref{FVarphi2}) holds true as well for any $\phi\in L^2(J^2)$
and any $y\in L^{p'}(\Omega)$. Hence $(s,t)\mapsto \Gamma(t)\varphi(s)$
belongs to $L^p(\Omega; L^2(J^2))$ and represents $(T\overline{\otimes} I_{L^2(J)})(\varphi)$.
Then the inequality (\ref{FVarphi1}) follows at once.
\end{proof}

\bigskip
We now turn to Rademacher averages and $R$-boundedness. Let $(\varepsilon_k)_{k\geq 1}$
be a sequence of independent Rademacher variables on some probability space $\Omega_0$.
For any Banach space $X$, we let ${\rm Rad}(X)\subset L^{2}(\Omega_0;X)$ denote the closed linear 
span of finite sums $\sum_k\varepsilon_k\otimes x_k\,$, with $x_k\in X$. 
A well-known application of Khintchine's inequality is that 
we have an isomorphism ${\rm Rad}(L^p(\Omega))\simeq 
L^p(\Omega;\ell^2)$. The argument for this result (see e.g. \cite[pp. 73-74]{LT})
shows as well that if $K$ is any Hilbert space, then we have
${\rm Rad}(L^p(\Omega;K))\simeq L^p(\Omega;\ell^2(K))$. Thus if
$(\varphi_k)_k$ is a finite family of functions $\Lambda\to L^p(\Omega)$ 
reprensenting elements of $L^p(\Omega;L^2(\Lambda))$, we have
\begin{equation}\label{Rad}
\Bignorm{\sum_k\varepsilon_k\otimes \varphi_k}_{{\rm Rad}(L^p(\Omega;L^2(\Lambda)))}
\,\approx\,\Bignorm{\Bigl(\int_{\Lambda}\sum_k\vert\varphi_k(t)\vert^2\, 
d\nu(t)\,\Bigr)^{\frac{1}{2}}}_{L^p(\Omega)}.
\end{equation}

By definition, a set $F\subset B(X)$ is $R$-bounded if there is a constant $C\geq 0$ such that
$$
\Bignorm{\sum_k\varepsilon_k\otimes T_k(x_k)}_{{\rm Rad}(X)}\,\leq\, C
\Bignorm{\sum_k\varepsilon_k\otimes x_k}_{{\rm Rad}(X)}
$$
for any finite families $(T_k)_{k}$ in $F$ and $(x_k)_{k}$ in $X$.

Let $A$ be a sectorial operator on $X$. We say that $A$ is $R$-sectorial of $R$-type
$\omega\in (0,\pi)$ if $\sigma(A)\subset\overline{\Sigma_{\omega}}$ and for any angle
$\theta\in (\omega,\pi)$, the set
$$
\bigl\{ z(z-A)^{-1}\, :\, z\in \Cdb\setminus \overline{\Sigma_{\theta}}\bigr\}
$$
is $R$-bounded. Clearly this condition is a strengthening of (\ref{Sector}).

$R$-boundedness goes back to \cite{BG,CPSW}, whereas
$R$-sectoriality was introduced by Weis \cite{W1}. That fundamental paper was the
starting point of a bunch of applications of $R$-boundedness to 
various questions involving multipliers, square functions
and functional calculi. See in particular 
\cite{KuW,KW1} and the references therein. 
We will use the following result (see \cite[Prop. 5.1]{KW1}) which
shows the role of $R$-boundedness in the problem of reducing
the angle of a bounded $H^\infty$ functional calculus.

\begin{proposition}\label{KW} (Kalton-Weis)
Let $1<p<\infty$, let $A$ be a sectorial operator on $L^p(\Omega)$
and let $0<\omega<\theta_0<\pi$ be two angles such that $A$
has a bounded $H^{\infty}(\Sigma_{\theta_0})$ functional calculus, and
$A$ is $R$-sectorial of $R$-type $\omega$. Then for any $\theta>\omega$, the
operator $A$ has a bounded $H^{\infty}(\Sigma_{\theta})$ functional calculus.
\end{proposition}

\section{Square function estimates imply $R$-sectoriality}

Our main result is the following. 

\begin{theorem}\label{Main2}
Let $A$ be a sectorial operator operator of type $<\frac{\pi}{2}$ on $L^p(\Omega)$, with
$1<p<\infty$. Assume that $A$ and $A^*$ satisfy square function estimates
\begin{equation}\label{SFE3}
\norm{x}_{A,\frac{1}{2}}\lesssim\norm{x}_p,\qquad x\in L^{p}(\Omega),
\end{equation}
and
\begin{equation}\label{SFE4}
\norm{y}_{A^*,\frac{1}{2}}\lesssim\norm{y}_{p'},\qquad y\in L^{p'}(\Omega).
\end{equation}
Then $A$ is $R$-sectorial of $R$-type $<\frac{\pi}{2}$.
\end{theorem}

According to Proposition \ref{KW} and the discussion before Theorem \ref{Main}, 
the latter is a consequence of Theorem \ref{Main2}.

\begin{proof}[Proof of Theorem \ref{Main2}]
We assume (\ref{SFE3}) and (\ref{SFE4}). 
Consider the two sets 
$$
F_1=\bigl\{T_t\, :\, t\geq 0\bigr\}
\qquad\hbox{and}\qquad
F_2=\bigl\{tAT_t\, :\, t\geq 0\bigr\}.
$$
According to \cite[Section 4]{W1}, showing that $A$ is $R$-sectorial 
of $R$-type $<\frac{\pi}{2}$ is equivalent to showing that $F_1$ and $F_2$ are $R$-bounded.

Before getting to the proof of these two properties, we need to establish a 
few intermediate estimates. We first show that the square function
$\norm{x}_{A,1}$ is finite for any $x\in L^p(\Omega)$ and that we have a
uniform estimate
\begin{equation}\label{SFE1}
\norm{x}_{A,1}\lesssim\norm{x}_p,\qquad
x\in L^p(\Omega).
\end{equation}
Indeed set $\Gamma(t)=A^{\frac{1}{2}}T_t$ for any $t>0$, consider
$x\in L^p(\Omega)$ and set $\varphi(s)=A^{\frac{1}{2}}T_s (x)$ for 
any $s>0$. Then
$$
\Gamma(t)\varphi(s) = A^{\frac{1}{2}}T_t\bigl(A^{\frac{1}{2}}T_s (x)\bigr)
= AT_{t+s}(x),\qquad t,s>0.
$$
Applying twice the estimate (\ref{SFE3}) together with Lemma \ref{FVarphi}, we deduce an estimate
$$
\Bignorm{\Bigl(\int_{0}^{\infty}\int_{0}^{\infty} \bigl\vert AT_{t+s}(x)\bigr\vert^2\, dsdt \,\Bigr)^{\frac{1}{2}}}_{L^p(\Omega)}\,
\lesssim\,\norm{x}_p,\qquad x\in L^p(\Omega).
$$
Now observe that at the pointwise level, we have
\begin{align*}
\int_{0}^{\infty} t\bigl\vert AT_t(x)\bigr\vert^2\,dt\ 
&=\,\int_{0}^{\infty} \Bigl(\int_{0}^{t}\, ds\,\Bigr) \bigl\vert AT_t(x)\bigr\vert^2\,dt\\
&=\,\int_{0}^{\infty}\int_{s}^{\infty} \bigl\vert AT_t(x)\bigr\vert^2\,dt\, ds\\
&=\,\int_{0}^{\infty}\int_{0}^{\infty} \bigl\vert AT_{t+s}(x)\bigr\vert^2\,dsdt\,.
\end{align*}
This yields (\ref{SFE1}).

Second we show that the square function
$\norm{x}_{A,\frac{3}{2}}$ is finite for any $x\in L^p(\Omega)$ and that we have a
uniform estimate
\begin{equation}\label{SFE32}
\norm{x}_{A,\frac{3}{2}}\lesssim\norm{x}_p,\qquad
x\in L^p(\Omega).
\end{equation}
Indeed using (\ref{SFE1}) and  (\ref{SFE3}), and arguing as above, we obtain 
an estimate
$$
\Bignorm{\Bigl(\int_{0}^{\infty}\int_{0}^{\infty} s \bigl\vert A^{\frac{3}{2}}
T_{t+s}x\bigr\vert^2\, dsdt \,\Bigr)^{\frac{1}{2}}}_{L^p(\Omega)}\,
\lesssim\,\norm{x}_p,\qquad x\in L^p(\Omega).
$$
This implies (\ref{SFE32}), since
\begin{align*}
\int_{0}^{\infty} t^2\bigl\vert A^{\frac{3}{2}} T_t(x)\bigr\vert^2\,dt\ 
&=\,2\,\int_{0}^{\infty} \Bigl(\int_{0}^{t}\,s\, ds\,\Bigr) \bigl\vert A^{\frac{3}{2}}
T_t(x)\bigr\vert^2\,dt\\
&=\, 2 \,\int_{0}^{\infty} s \int_{s}^{\infty} \bigl\vert A^{\frac{3}{2}}
T_t(x)\bigr\vert^2\,dt\, ds\\
&=\,2 \,\int_{0}^{\infty}\int_{0}^{\infty}  s \bigl\vert A^{\frac{3}{2}}
T_{t+s}(x)\bigr\vert^2\,dsdt\,.
\end{align*}

In the sequel, we let $R(A)$ and $N(A)$ denote the range and the kernel of
$A$, respectively. It is well-known that we have a direct sum decomposition
\begin{equation}\label{Decomp}
L^{p}(\Omega) = N(A)\oplus\overline{R(A)},
\end{equation}
see e.g. \cite[Thm 3.8]{CDMY}.

We observe that the dual estimate (\ref{SFE4}) implies 
a uniform reverse estimate
\begin{equation}\label{Reverse}
\norm{x}_p\lesssim \norm{x}_{A,\frac{1}{2}},\qquad
x\in \overline{R(A)}.
\end{equation}
This follows from a well-known duality argument, we briefly give a proof for the sake of 
completeness. For any integer $n\geq 1$, let $g_n$ be the rational function
defined by $g_n(z)=n^2 z(n+z)^{-1}(1+nz)^{-1}$. 
For any $x\in L^p(\Omega)$, we have
$$
g_n(A)x=\,2\,\int_{0}^{\infty} AT_{2t}g_n(A)x\, dt,
$$
by \cite[Lem. 6.5 (1)]{JLX}. Moreover the sequence $\bigl(g_n(A)\bigr)_{n\geq 1}$
is bounded. Hence for any 
$y$ in $L^{p'}(\Omega)$, we have
\begin{align*}
\frac{1}{2}\,\bigl\vert \langle g_n(A)x,y\rangle\bigr\vert \,
&=\,\Bigl\vert \int_{0}^{\infty} \bigl\langle AT_{2t}g_n(A)x,y\bigr\rangle\, dt\,\Bigr\vert\\
&=\,\Bigl\vert \int_{0}^{\infty} \bigl\langle A^{\frac{1}{2}} T_{t}(x),
A^{*\frac{1}{2}} T_{t}^{*}g_n(A)^*y\bigr\rangle\, dt\,\Bigr\vert\\
&\leq\,\norm{x}_{A,\frac{1}{2}}\norm{g_n(A)^*y}_{A^*,\frac{1}{2}}\qquad\hbox{by Cauchy-Schwarz},\\
&\lesssim\, 
\norm{x}_{A,\frac{1}{2}}\norm{g_n(A)^*y}_{p'} \,\lesssim
\norm{x}_{A,\frac{1}{2}}\norm{y}_{p'}\qquad\hbox{by (\ref{SFE4})}.
\end{align*}
If $x\in \overline{R(A)}$, then $g_n(A)x\to x$ when $n\to\infty$ (see \cite[Thm. 3.8]{CDMY}),
hence (\ref{Reverse}) follows by taking the supremum over all $y$ in the 
unit ball of $L^{p'}(\Omega)$.

Let $(x_k)_k$ be a finite family of $\overline{R(A)}$ and for any $k$, set $\varphi_k(t)=A^{\frac{1}{2}}T_t(x_k)$
for any $t>0$. Averaging (\ref{Reverse}), we obtain that
$$
\Bignorm{\sum_k\varepsilon_k\otimes x_k}_{{\rm Rad}(L^p(\Omega))}
\,\lesssim\,\Bignorm{\sum_k\varepsilon_k\otimes \varphi_k}_{{\rm Rad}(L^p(\Omega;L^2(J)))}.
$$
Then applying (\ref{Rad}) we deduce a
uniform estimate
\begin{equation}\label{Rad-Reverse}
\Bignorm{\sum_k\varepsilon_k\otimes x_k}_{{\rm Rad}(L^p(\Omega))}
\,\lesssim\,
\Bignorm{\Bigl(\sum_k\int_{0}^{\infty}\bigl\vert 
A^{\frac{1}{2}}T_t(x_k)\bigr\vert^2\, dt\,\Bigr)^{\frac{1}{2}}}_{L^p(\Omega)}
\end{equation}
for finite families $(x_k)_k$ of elements of $\overline{R(A)}$. 

By the same averaging principle, the assumption (\ref{SFE3}) implies  a
uniform estimate
\begin{equation}\label{Rad-SFE12}
\Bignorm{\Bigl(\sum_k\int_{0}^{\infty}\bigl\vert 
A^{\frac{1}{2}}T_t(x_k)\bigr\vert^2\, dt\,\Bigr)^{\frac{1}{2}}}_{L^{p}(\Omega)}
\,\lesssim\,
\Bignorm{\sum_k\varepsilon_k\otimes x_k}_{{\rm Rad}(L^p(\Omega))}
\end{equation}
for finite families $(x_k)_k$ of elements of $L^p(\Omega)$.

Likewise, (\ref{SFE32}) implies that we have a uniform estimate
\begin{equation}\label{Rad-SFE32}
\Bignorm{\Bigl(\sum_k\int_{0}^{\infty} t^2\bigl\vert 
A^{\frac{3}{2}}T_t(x_k)\bigr\vert^2\, dt\,\Bigr)^{\frac{1}{2}}}_{L^{p}(\Omega)}
\,\lesssim\,
\Bignorm{\sum_k\varepsilon_k\otimes x_k}_{{\rm Rad}(L^p(\Omega))}
\end{equation}
for finite families $(x_k)_k$ of elements of $L^p(\Omega)$.

We now turn to $R$-boundedness proofs. Let $(t_k)_k$ be a 
finite family of nonnegative real numbers. 

For any $x_1, x_2,\ldots$ in 
$\overline{R(A)}$, we have $T_{t_k}(x_k)\in \overline{R(A)}$ for any $k$, hence
$$
\Bignorm{\sum_k\varepsilon_k\otimes T_{t_k}(x_k)}_{{\rm Rad}(L^p(\Omega))}
\,\lesssim\,
\Bignorm{\Bigl(\sum_k\int_{0}^{\infty}\bigl\vert A^{\frac{1}{2}}T_{t+t_k}
(x_k)\bigr\vert^2\, dt\,\Bigr)^{\frac{1}{2}}}_{L^p(\Omega)}
$$
by (\ref{Rad-Reverse}).
Moreover we have
\begin{align*}
\Bignorm{\Bigl(\sum_k\int_{0}^{\infty}\bigl\vert A^{\frac{1}{2}}T_{t+t_k}
(x_k)\bigr\vert^2\, dt\,\Bigr)^{\frac{1}{2}}}_{L^p(\Omega)}\,
&=\,
\Bignorm{\Bigl(\sum_k\int_{t_k}^{\infty}\bigl\vert A^{\frac{1}{2}}T_{t}
(x_k)\bigr\vert^2\, dt\,\Bigr)^{\frac{1}{2}}}_{L^p(\Omega)}\\
&\leq\,
\Bignorm{\Bigl(\sum_k\int_{0}^{\infty}\bigl\vert A^{\frac{1}{2}}T_{t}
(x_k)\bigr\vert^2\, dt\,\Bigr)^{\frac{1}{2}}}_{L^p(\Omega)}.
\end{align*}
Applying (\ref{Rad-SFE12}), we deduce an estimate
$$
\Bignorm{\sum_k\varepsilon_k\otimes T_{t_k}(x_k)}_{{\rm Rad}(L^p(\Omega))}
\,\lesssim\,
\Bignorm{\sum_k\varepsilon_k\otimes  x_k}_{{\rm Rad}(L^p(\Omega))}.
$$
Since $T_t(x)=x$ for any $x\in N(A)$, the above estimate 
and (\ref{Decomp}) show that the set $F_1$ is $R$-bounded.

Next consider $x_1, x_2,\ldots$ in $L^p(\Omega)$. Applying (\ref{Rad-Reverse}) again, we have
\begin{align*}
\Bignorm{\sum_k\varepsilon_k\otimes t_kAT_{t_k}(x_k)}_{{\rm Rad}(L^p(\Omega))}
\,&\lesssim\,\Bignorm{\Bigl(\sum_k\int_{0}^{\infty}t_k^2\bigl\vert A^{\frac{3}{2}}T_{t+t_k}
(x_k)\bigr\vert^2\, dt\,\Bigr)^{\frac{1}{2}}}_{L^p(\Omega)}\\
&\lesssim\,\Bignorm{\Bigl(\sum_k\int_{0}^{\infty}(t+t_k)^2\bigl\vert A^{\frac{3}{2}}T_{t+t_k}
(x_k)\bigr\vert^2\, dt\,\Bigr)^{\frac{1}{2}}}_{L^p(\Omega)}\\
&\lesssim\,\Bignorm{\Bigl(\sum_k\int_{t_k}^{\infty}t^2\bigl\vert A^{\frac{3}{2}}T_{t}
(x_k)\bigr\vert^2\, dt\,\Bigr)^{\frac{1}{2}}}_{L^p(\Omega)}\\
&\lesssim\,\Bignorm{\Bigl(\sum_k\int_{0}^{\infty}t^2\bigl\vert A^{\frac{3}{2}}T_{t}
(x_k)\bigr\vert^2\, dt\,\Bigr)^{\frac{1}{2}}}_{L^p(\Omega)}.
\end{align*}
According to (\ref{Rad-SFE32}), this implies the estimate
$$
\Bignorm{\sum_k\varepsilon_k\otimes t_kAT_{t_k}(x_k)}_{{\rm Rad}(L^p(\Omega))}
\,\lesssim\,\Bignorm{\sum_k\varepsilon_k\otimes x_k}_{{\rm Rad}(L^p(\Omega))},
$$
which shows that the set $F_2$ is $R$-bounded.
\end{proof}

\section{Concluding remarks}

\noindent
{\it 4.1 Comparing square functions estimates.}
Let $A$ be a sectorial operator of type $<\frac{\pi}{2}$ on $L^p(\Omega)$, with
$1<p<\infty$. 
If $A$ is $R$-sectorial of $R$-type $<\frac{\pi}{2}$, then 
the square functions $\norm{\ }_{A,\alpha}$ defined by (\ref{Alpha}) are
pairwise equivalent, by \cite[Thm. 1.1]{LM2}. We do not know if this
equivalence property holds in the general (non $R$-sectorial) case.
The proof of Theorem \ref{Main2} shows that if $A$ satisfies
an estimate $\norm{x}_{A,\frac{1}{2}}\lesssim \norm{x}_p$
on $L^p(\Omega)$, then it also satisfies estimates 
$\norm{x}_{A,1}\lesssim \norm{x}_p$ and 
$\norm{x}_{A,\frac{3}{2}}\lesssim \norm{x}_p$, which is a step towards that direction. It
is easy to check (left to the reader) that with the
same techniques, one obtains the following:
for any positive numbers $\alpha,\beta>0$, 
square function estimates
$$
\norm{x}_{A,\alpha}\lesssim \norm{x}_p
\qquad\hbox{and}\qquad
\norm{x}_{A,\beta}\lesssim \norm{x}_p
$$
imply a square function estimate
$$
\norm{x}_{A,\alpha+\beta}\lesssim \norm{x}_p.
$$

\smallskip
\noindent
{\it 4.2 Variants of the main result.}
Using the above observation, the proof of Theorem \ref{Main2} can be easily adapted
to show the following generalization: let $q\geq 1$ be an integer, let $\beta>0$ be a 
positive real number and assume that $A$ and $A^*$ satisfy square function estimates
$$
\norm{x}_{A,\frac{1}{q}}\lesssim \norm{x}_p
\qquad\hbox{and}\qquad
\norm{y}_{A^*,\beta}\lesssim \norm{y}_{p'}.
$$
Then $A$ is $R$-sectorial of $R$-type $<\frac{\pi}{2}$.

This implies that Theorem \ref{Main} holds as well with 
(\ref{SFE}) replaced by
$$
\norm{x}_{A,1}\lesssim\norm{x}_p
\qquad\hbox{and}\qquad
\norm{y}_{A^*,1}\lesssim\norm{y}_{p'}.
$$

\smallskip
\noindent
{\it 4.3 Other Banach spaces.} Any bounded set of operators on Hilbert space is automatically
$R$-bounded.
In the case $p=2$ (more generally, for sectorial operators on Hilbert space), 
Theorem \ref{Main} reduces to McIntosh's fundamental Theorem \cite{MI}.

In the last decade, various square functions associated to sectorial operators 
on general Banach spaces (not only on Hilbert spaces or $L^p$-spaces) have been 
studied in relation with $H^\infty$ functional calculus, see \cite{KW2}, \cite{KKW}
\cite{LM3} and the references therein. It is therefore
natural to wonder whether Theorem \ref{Main} can be extended to 
other contexts.

In a positive direction, we note that if $X$ is a reflexive Banach lattice  
and if square functions associated to sectorial
operators are defined as in (\ref{SFE}) then Theorem \ref{Main} holds 
true on $X$. This actually extends (for adapted square functions)
to the case when $X$ is a reflexive Banach space with property $(\alpha)$.
We refer the reader to \cite{LM4} for more on that theme.

However we do not know if Theorem \ref{Main} holds true on noncomutative $L^p$-spaces.
See \cite{JLX} for a thorough study of square functions associated to sectorial
operators on those spaces.

\end{document}